\documentclass{article}

\usepackage{amsmath}
\usepackage{graphicx}
\usepackage{ upgreek }
\usepackage{ wasysym }
\usepackage{amsthm}
\usepackage{ dsfont }

\newtheorem{theorem}{Theorem}
\newtheorem{definition}{Definition}
\newtheorem{lemma}{Lemma}
\newtheorem{conjecture}{Conjecture}

\title{Maximizing Volume Ratios for Shadow Covering by Tetrahedra}
\author{Christina Chen}
\date{November 28, 2011}

\begin{document}
\maketitle

\begin{abstract}
Define a body $A$ to be able to hide behind a body $B$ if the orthogonal projection of $B$ contains a translation of the corresponding orthogonal projection of $A$ in every direction.  In two dimensions, it is easy to observe that there exist two objects such that one can hide behind another and have a larger area than the other.  It was recently shown that similar examples exist in higher dimensions as well.  However, the highest possible volume ratio for such bodies is still undetermined.  We investigated two three-dimensional examples, one involving a tetrahedron and a ball and the other involving a tetrahedron and an inverted tetrahedron.  We calculate the highest volume ratio known up to this date, 1.16, which is generated by our second example.  

\end{abstract}

\section{Introduction}

The shadows of two-dimensional objects are line segments, so it is easy to construct examples of two objects such that one hides behind the other but has a larger area than the other.  However, in three dimensions, since shadows have shapes, the question is more complicated.  In \cite{Klain1}, it was shown that in three dimensions, there exist two convex bodies such that one can hide behind the other but have a larger volume than the other.  Examples that invoke Minkowski sums were presented in \cite{Klain3}, but optimal ratios were not calculated.  Recently, in \cite{Klain4}, it was shown that in $n$ dimensions, the optimal volume ratio for two such objects is bounded above by $(\frac{n}{n-1})^{n}$, which decreases monotonically and approaches $e\approx2.718$ as $n$ approaches infinity.  In particular, a universal bound of 2.942 was obtained, which is the value of the expression for $n = 7$ and which can be easily verified for $n \leq 6$.    

Here, we present two examples and calculate the optimal numerical ratios for each.  In particular, our examples require the Minkowski sums of a tetrahedron and a ball and a tetrahedron and an inverted tetrahedron.  The optimal volume ratio calculated for the first example is 1.12, and the optimal volume ratio calculated for the second example is 1.16.    

In Section~\ref{sec:def}, we define all the relevant terms and present two two-dimensional examples of pairs of objects such that one can hide behind another but have a bigger area than the other.  In Section~\ref{sec:tetraball} and Section~\ref{sec:tetratetra}, we present our new examples.  In Section~\ref{sec:futureresearch}, we conclude with ideas for continued research.

\section{Definitions}\label{sec:def}

\subsection{Hiding Behind}
\begin{definition}
Define a \emph{shadow} of an $n$-dimensional body $K$ as an orthogonal projection of $K$ onto $\mathds{R}^{n-1}$.  A body $A$ is defined to be able to \emph{hide behind} a body $B$ if the orthogonal projection of $B$ contains a translation of the corresponding orthogonal projection of $A$ in every direction.  
\end{definition}
For example, in Figure~\ref{fig:hidebehind}, a disk of radius $\frac{\sqrt{3}}{4}$ can hide behind a unit equilateral triangle.  The shadows of the disk are line segments of length $\frac{\sqrt{3}}{2}$, and the shadows of the triangle are line segments ranging in length from $\frac{\sqrt{3}}{2}$ to 1.

\begin{center}
\begin{figure}[http]
\centering
\includegraphics[scale=0.5]{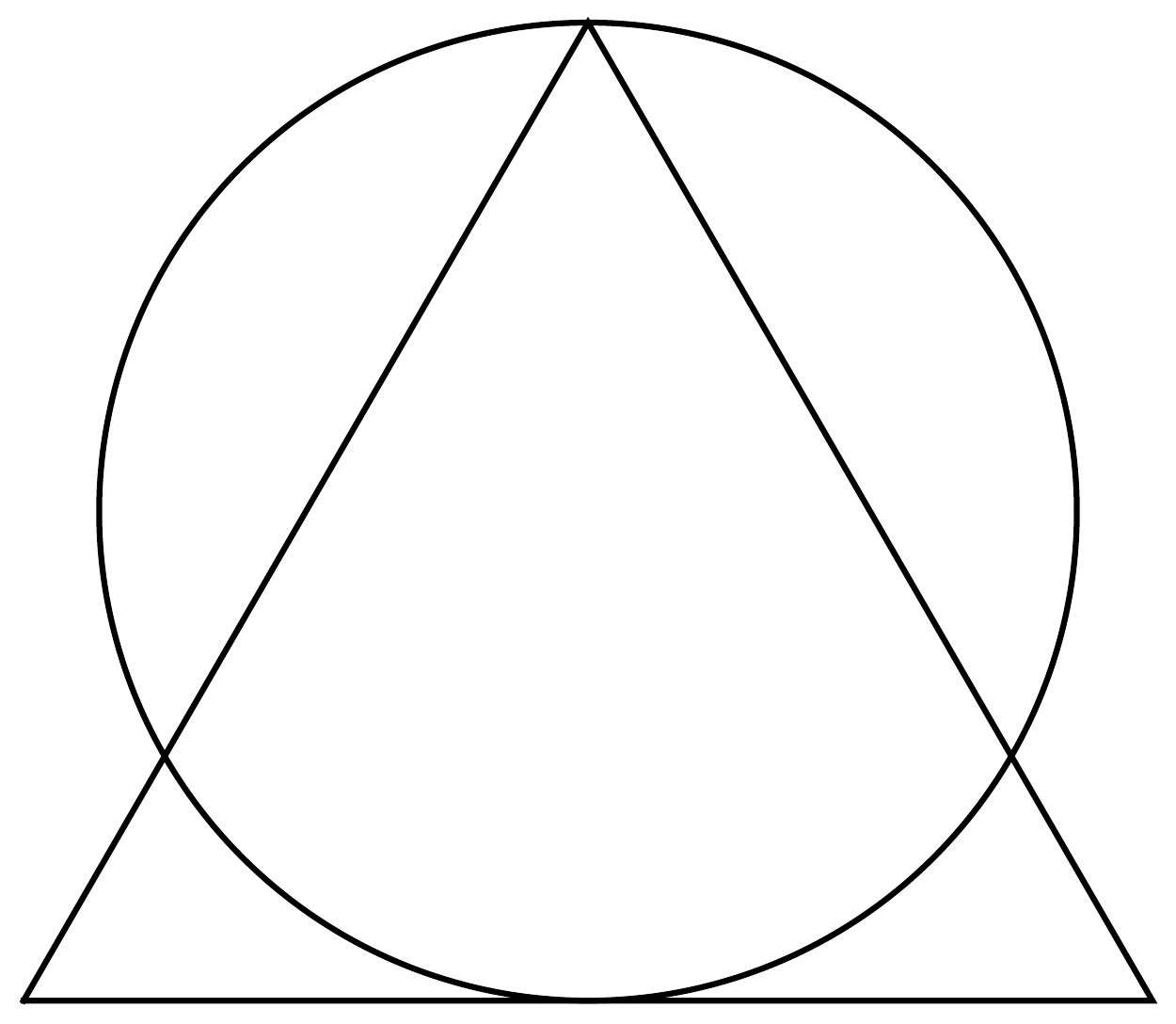}
\caption{The disk hides behind the triangle.}
\label{fig:hidebehind}
\end{figure}
\end{center}


\subsection{Minkowski Sum}

\begin{definition}
The \emph{Minkowski sum} of two sets $A$ and $B$ in a vector space is the set obtained by adding every element in $A$ to every element in $B$, expressed as
$$A+B=\{a+b \mid a\in A, b\in B\}.$$  
\end{definition}
For example, Figure~\ref{fig:minkexample} illustrates the Minkowsi sum of a square and an equilateral triangle, both scaled by $\frac{1}{2}$.  

\begin{center}
\begin{figure}[http]
\includegraphics[scale=0.95]{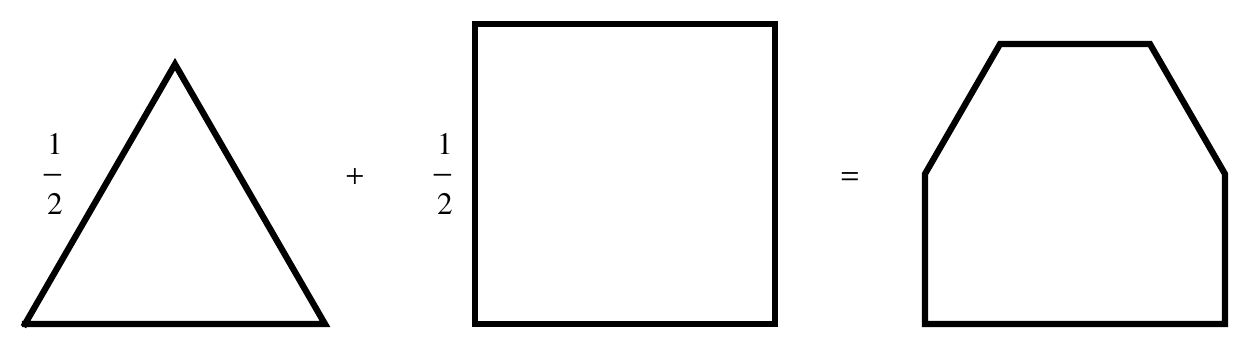}
\caption{The Minkowski sum of a scaled equilateral triangle and a scaled square.}
\label{fig:minkexample}
\end{figure}
\end{center}

We can also define the \emph{Minkowski interpolation} of two objects $A$ and $B$ as the Minkowski sum of $\alpha A$ and $\beta B$, where $\alpha$ and $\beta$ are scaling constants.

\begin{lemma}\label{lem:hidingsum}
For convex bodies $A$, $B$, and $C$, if $A$ and $B$ can hide behind $C$, then for any $\upmu$ such that $0 \leq \upmu \leq 1$, $\upmu A+(1-\upmu)B$ can hide behind $C$. 
\end{lemma}

\begin{proof}  

For an $n$-dimensional body $A$ and a unit vector $u$ in $\mathds{R}^{n}$, define $A_{u}$ to be the projection of $A$ onto the subspace $u^{\perp}$.  Since $A$ and $B$ can hide behind $C$, there exist vectors $x, y \in u^\perp$ such that $A_u + x \subseteq C_u$ and $B_u + y \subseteq C_u$.
It follows that 

\begin{align*}
\upmu A_u + (1-\upmu)B_u + [\upmu x + (1-\upmu) y] 
&= \upmu (A_u + x) + (1-\upmu)(B_u + y)\\
&\subseteq \upmu C_u + (1-\upmu)C_u = C_u.
\end{align*}
Since the projection of the sum is the sum of the projections, 
$$[\upmu A + (1-\upmu)B]_u = \upmu A_u + (1-\upmu)B_u,$$
so $\upmu A + (1-\upmu)B$ can also hide behind $C$. 

\end{proof}

\begin{lemma}\label{lem:scaling}
For a scaling constant $s$, $0 < s < 1$, such that $sA$ can hide behind $B$, and any constant $t$, $0 < t < 1$, the volume of the Minkowski sum $tsA+(1-t)B$ is maximized when $s$ is maximized. 
 
\end{lemma} 

\begin{proof}
Invoking mixed volumes, the volume of the Minkowski sum can be expressed as $$t^{3}s^{3}V(A)+2t^{2}s^{2}(1-t)V(A, A, B)+3ts(1-t)^{2}V(A, B, B)+(1-t)^{3}V(B).$$
Note that this is a polynomial expression with respect to $s$ with positive coefficients.  It follows that this expression is maximized when $s$ is maximized.
\end{proof}

\subsection{Two-Dimensional Examples}

In two dimensions, the Minkowski interpolation of a triangle and a disk can hide behind the original triangle but not hide inside it.  See Figure~\ref{fig:triangledisk}.   
 
\begin{center}
\begin{figure}[http]
\includegraphics[scale=0.95]{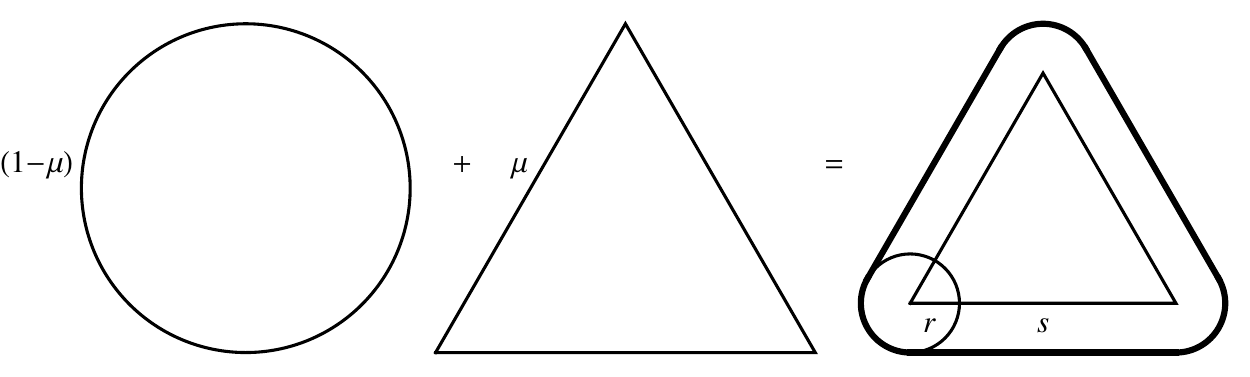}
\caption{Minkowski sum of triangle and disk.}
\label{fig:triangledisk}
\end{figure}
\end{center}

Although the original disk, with radius $\frac{\sqrt{3}}{4}$, which has a larger area than the unit triangle, can hide behind the triangle, invoking Minkowski sums generates a shape with a greater area ratio.  More precisely, the area ratio is maximized by the Minkowski sum of the disk scaled by $\frac{6-\sqrt{3}\pi}{8-\sqrt{3}\pi}\approx 0.22$, generating the approximate area ratio $1.39$. 

As another example, the Minkowski interpolation of a triangle and an inverted triangle can hide behind the original disk but not hide inside it.  See Figure~\ref{fig:triangletriangle}.  The area ratio is maximized by the Minkowski sum of both shapes scaled by $\frac{1}{2}$, generating the area ratio $1.5$, which is the optimal area ratio in two dimensions, by Steiner's Theorem. 

\begin{center}
\begin{figure}[http]
\includegraphics[scale=0.95]{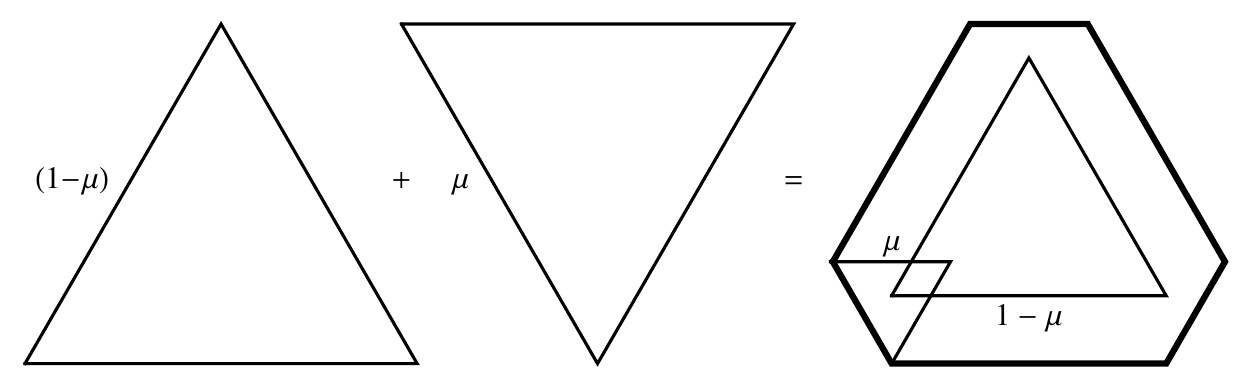}
\caption{Minkowski sum of triangle and inverted triangle.}
\label{fig:triangletriangle}
\end{figure}
\end{center}

Note that the case involving two triangles generates a higher area ratio than the case involving a triangle and a disk.  This result extends to three dimensions, as will be shown in the next two sections, which describe two three-dimensional examples involving a tetrahedron. 

\section{Tetrahedron and Ball}\label{sec:tetraball}

One example in which one body can hide behind another body of smaller volume is generated by the Minkowski interpolation of a tetrahedron and a ball.  Suppose a unit tetrahedron is positioned in the coordinate plane with vertices at $(0, 0, 0)$, $(1, 0, 0)$, $(\frac{1}{2}, \frac{\sqrt{3}}{2},0)$, $(\frac{1}{2}, \frac{\sqrt{3}}{6}, \sqrt{\frac{2}{3}})$.  By Lemma~\ref{lem:scaling}, to calculate the highest volume ratio in this example, it is necessary to first calculate the radius of the largest circle that can always be contained in the shadow of the tetrahedron, in other words, the smallest inradius of all possible orthogonal projections of the tetrahedron.

\begin{lemma}\label{lem:inradius}
The smallest inradius of all possible shadows of the tetrahedron is $\frac{\sqrt{6}-\sqrt{2}}{4}$. 
\end{lemma}

\begin{proof}
All possible orthogonal projections onto the $xy$-plane can be obtained by rotating the tetrahedron first clockwise about the $x$-axis by $\alpha$ degrees and then counterclockwise about the $y$-axis by $\beta$ degrees, which is sufficient because only the shape and not the configuration of the projection is relevant, as all the orthogonal projections of a ball are disks. The new coordinates of the vertices of the tetrahedron after these transformations can be explicitly calculated by matrix multiplication involving the rotation matrices
$$\begin{pmatrix}
1&0&0\\
0&\cos\alpha&-\sin\alpha\\
0&\sin\alpha&\cos\alpha\\
\end{pmatrix} \text{ and }
\begin{pmatrix}
\cos\beta&0&\sin\beta\\
0&1&0\\
-\sin\beta&0&\cos\beta\\
\end{pmatrix}.$$
For example, for the vertex $(1, 0, 0)$, we have
$$
\begin{pmatrix}
\cos\beta&0&\sin\beta\\
0&1&0\\
-\sin\beta&0&\cos\beta\\
\end{pmatrix}
\begin{pmatrix}
1&0&0\\
0&\cos\alpha&-\sin\alpha\\
0&\sin\alpha&\cos\alpha\\
\end{pmatrix}
\begin{pmatrix}
1\\
0\\
0\\
\end{pmatrix}
=
\begin{pmatrix}
\cos\beta\\
0\\
-\sin\beta\\
\end{pmatrix}.\\
$$
Therefore, after the rotations, the coordinates of the projections of the vertices of the unit tetrahedron onto the $xy$-plane are 
$$(0, 0, 0)\longrightarrow(0, 0)$$\\
$$(1, 0, 0)\longrightarrow(\cos\beta, 0)$$\\
$$(\frac{1}{2}, \frac{\sqrt{3}}{2}, 0)\longrightarrow(\frac{1}{2}\cos\beta+\frac{\sqrt{3}}{2}\sin\alpha\sin\beta, \frac{\sqrt{3}}{2}\cos\alpha)$$\\
$$(\frac{1}{2}, \frac{\sqrt{3}}{6}, \sqrt{\frac{2}{3}})\longrightarrow(\frac{1}{2}\cos\beta+\frac{\sqrt{3}}{6}\sin\alpha\sin\beta+\sqrt{\frac{2}{3}}\cos\alpha\sin\beta, \frac{\sqrt{3}}{6}\cos\alpha-\sqrt{\frac{2}{3}}\sin\alpha).$$\\
The shadow obtained by any projection is defined by the convex hull of these coordinates.

Employing this procedure, we wrote a program that verified that the radius of the largest disk that can be contained in all orthogonal projections of the tetrahedron is $\frac{\sqrt{6}-\sqrt{2}}{4}$, which is obtained by projecting parallel to one of the faces of the tetrahedron.  

\end{proof}

\begin{theorem}

The maximum volume ratio for the Minkowski interpolation of a tetrahedron and a ball is obtained by scaling the ball by
$$\frac{3\sqrt{2}-\sqrt{6}+c_{1}\pi+(12-6\sqrt{3})\alpha}{8\sqrt{2}-3\sqrt{6}+c_{2}\pi+(18-9\sqrt{3})\alpha}+\sqrt{\frac{c_{3}+c_{4}\pi+c_{5}\pi^2+c_{6}\alpha+c_{7}\pi\alpha+c_{8}\alpha^2}{8\sqrt{2}-3\sqrt{6}+c_{9}\pi+18\alpha-9\sqrt{3}\alpha^2}}$$

$$\approx 0.68424688,$$  
where

\begin{align*}c_{1}&=6\sqrt{3}+3\sqrt{6}-5\sqrt{2}-12\\
c_{2}&=9\sqrt{3}+3\sqrt{6}-5\sqrt{2}-18\\
c_{3}&=24-12\sqrt{3}+18\sqrt{6}\\
c_{4}&=38-33\sqrt{2}-22\sqrt{3}\\
c_{5}&=63-36\sqrt{3}\\
c_{6}&=33\sqrt{2}-18\sqrt{6}\\
c_{7}&=72\sqrt{3}-126\\
c_{8}&=63-36\sqrt{3}\\
c_{9}&=9\sqrt{3}+3\sqrt{6}-18-5\sqrt{2},\\
\end{align*}
and the corresponding volume ratio is approximately 1.124358246.

\end{theorem}

\begin{proof}
Let $\triangle$ and $\ocircle$ denote the original tetrahedron and the ball with radius $\frac{\sqrt{6}-\sqrt{2}}{4}$, respectively.  Then, the Minkowski interpolation 
$$\upmu \triangle + (1-\upmu)\ocircle,$$\\
where $0 \leq \upmu \leq 1$, can hide behind the tetrahedron.  Its volume can be expressed as 
$$\frac{\sqrt{2}}{12}s^3+3(\pi-\alpha)sr^2+s\sqrt{3}s^2r+\frac{4}{3}\pi r^3,$$
where 
$$r=\frac{\sqrt{6}-\sqrt{2}}{4}(1-\upmu)$$\\
$$s=\upmu$$\\
$$\alpha=\cos^{-1}\left(\frac{1}{3}\right).$$\\
Expanding this expression yields a cubic function with respect to $\upmu$, which is maximized at

$$\upmu=\frac{3\sqrt{2}-\sqrt{6}+c_{1}\pi+(12-6\sqrt{3})\alpha}{8\sqrt{2}-3\sqrt{6}+c_{2}\pi+(18-9\sqrt{3})\alpha}+\sqrt{\frac{c_{3}+c_{4}\pi+c_{5}\pi^2+c_{6}\alpha+c_{7}\pi\alpha+c_{8}\alpha^2}{8\sqrt{2}-3\sqrt{6}+c_{9}\pi+18\alpha-9\sqrt{3}\alpha^2}}$$

$$\approx 0.68424688,$$  
where

\begin{align*}c_{1}&=6\sqrt{3}+3\sqrt{6}-5\sqrt{2}-12\\
c_{2}&=9\sqrt{3}+3\sqrt{6}-5\sqrt{2}-18\\
c_{3}&=24-12\sqrt{3}+18\sqrt{6}\\
c_{4}&=38-33\sqrt{2}-22\sqrt{3}\\
c_{5}&=63-36\sqrt{3}\\
c_{6}&=33\sqrt{2}-18\sqrt{6}\\
c_{7}&=72\sqrt{3}-126\\
c_{8}&=63-36\sqrt{3}\\
c_{9}&=9\sqrt{3}+3\sqrt{6}-18-5\sqrt{2},\\
\end{align*}
generating the maximum volume approximately 0.13250689 and the volume ratio approximately 1.124358246.

\end{proof}

\section{Tetrahedron and Inverted Tetrahedron}\label{sec:tetratetra}

Another example in which one body can hide behind another body of smaller volume is generated by the Minkowski sum of a tetrahedron and an inverted tetrahedron.  See Figure~\ref{fig:uptetra}, Figure~\ref{fig:downtetra}, and Figure~\ref{fig:tetratetra} for a tetrahedron, an inverted tetrahedron, and the Minkowski sum obtained by scaling the tetrahedrons by some constants $s$ and $1-s$. 

\begin{center}
\begin{figure}[ht]
\centering
\includegraphics[scale=0.5]{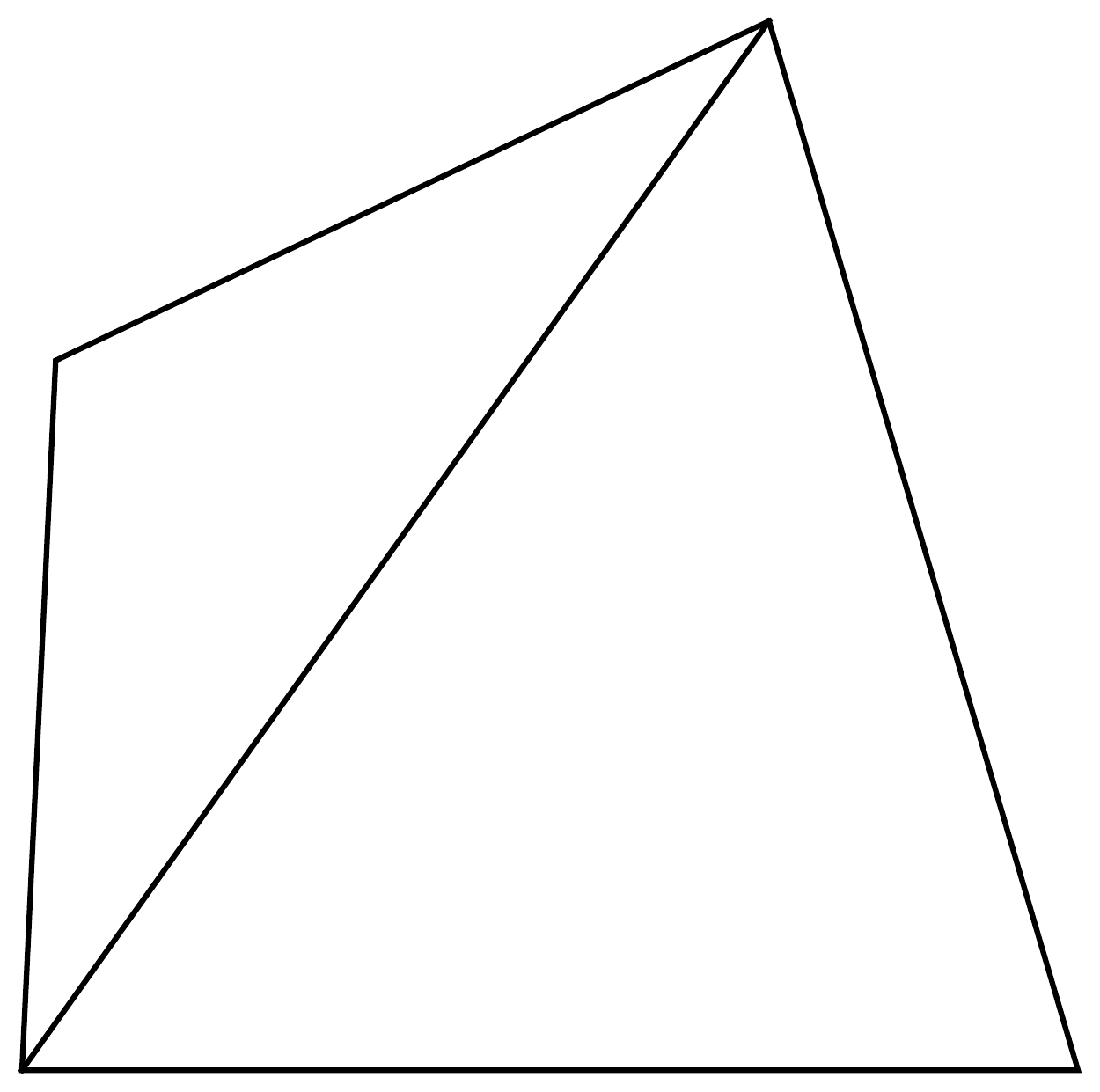}
\caption{A tetrahedron rotated.}
\label{fig:uptetra}
\end{figure}
\end{center}

\begin{center}
\begin{figure}[http]
\centering
\includegraphics[scale=0.5]{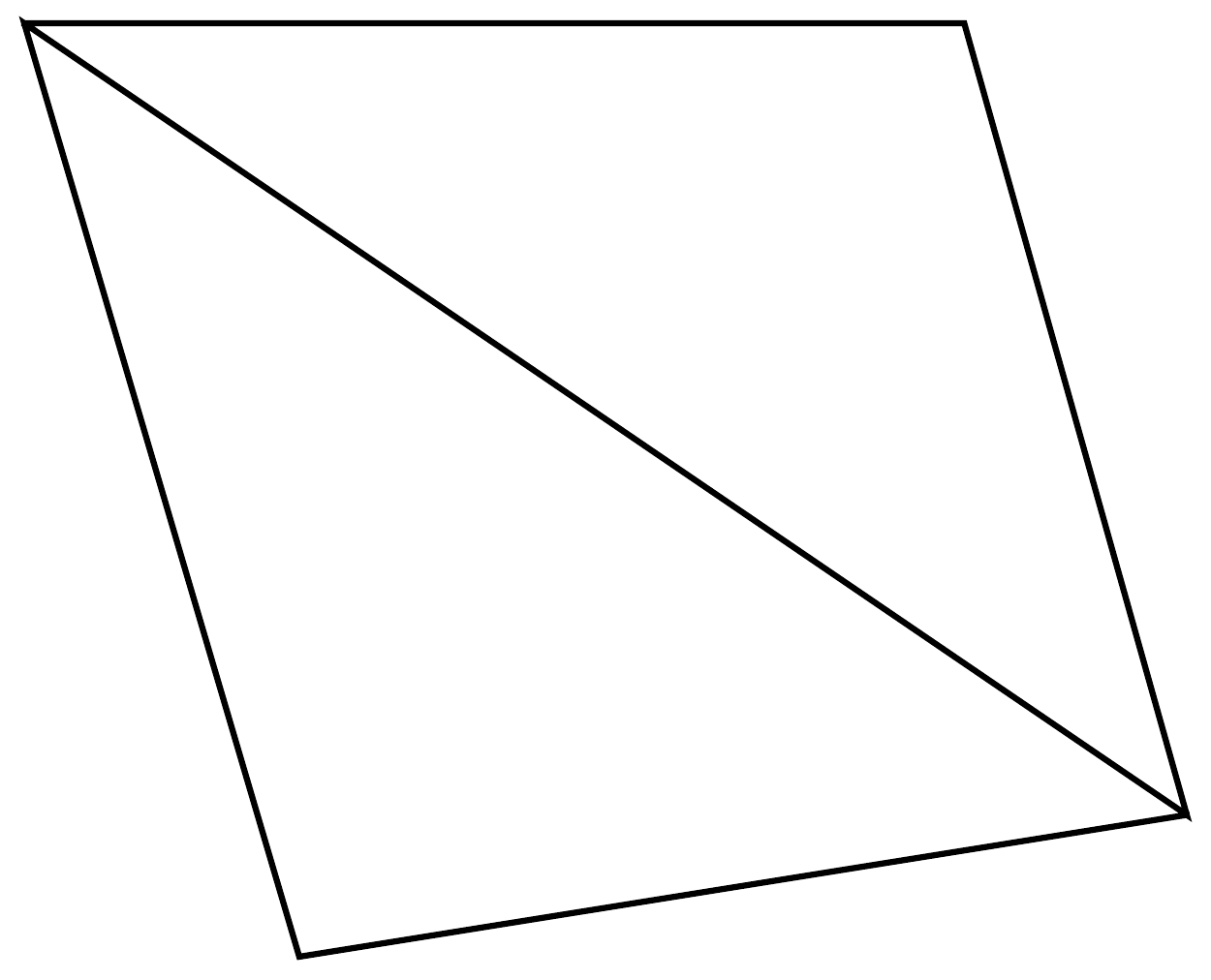}
\caption{A scaled and inverted tetrahedron.}
\label{fig:downtetra}
\end{figure}
\end{center}

\begin{center}
\begin{figure}[http]
\centering
\includegraphics[scale=0.5]{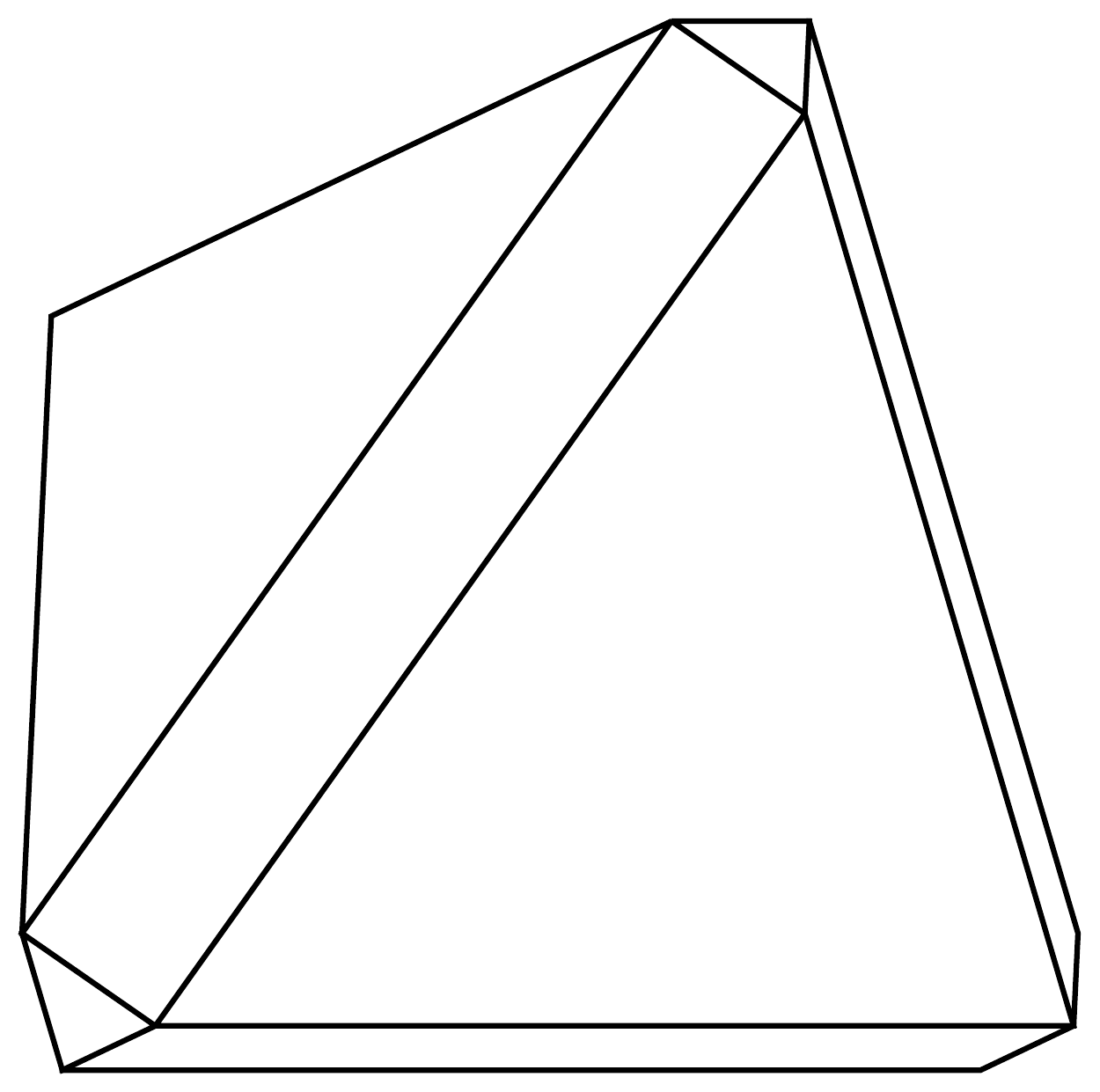}
\caption{Minkowski interpolation of tetrahedron and inverted tetrahedron.}
\label{fig:tetratetra}
\end{figure}
\end{center}

\begin{lemma}\label{lem:side lenth}
The largest inverted tetrahedron whose orthogonal projections can always be contained in the corresponding shadows of a unit tetrahedron has side length $\frac{1}{2}$. 
\end{lemma}

\begin{proof}
By Lutwak's Containment Theorem in \cite{Lutwak1} and in \cite{Klain2}, in two dimensions, any shape $K$ contains a translate of $-\frac{1}{2}K$.  Since the shadows of $\triangle$ and $-\triangle$ are two-dimensional, it follows that each projection of $\triangle$ contains a translate of the corresponding projection of $-\frac{1}{2}\triangle$, so $-\frac{1}{2}\triangle$ can hide behind $\triangle$.

To prove that this bound is sharp, note the projection perpendicular to the bases of the tetrahedra, which generates equilateral triangle shadows.  It is clear that $-\frac{1}{2}\triangle$ can hide behind $\triangle$ but $-x\triangle$ cannot hide behind $\triangle$ for any $x > \frac{1}{2}$. 
\end{proof}

\begin{theorem}

The maximum volume ratio for the Minkowski interpolation of a tetrahedron and an inverted tetrahedron is obtained by scaling the inverted tetrahedron by $\frac{1+2\sqrt{14}}{11}\approx0.77121003$, and the corresponding volume ratio is approximately 1.1633587.

\end{theorem}

\begin{proof}

For the appropriate parameters $\alpha$ and $\beta=\frac{1-\alpha}{2}$, the volume of the Minkowski interpolation of the tetrahedron and the inverted tetrahedron can be expressed, by means of mixed volumes, as
\begin{align*}
&V[\alpha\triangle+\beta(-\triangle)]\\
&=\alpha^3V(\triangle)+3\alpha^2\beta V(\triangle, \triangle, -\triangle)+3\alpha\beta^2V(\triangle, -\triangle, -\triangle)+\beta^3V(-\triangle).
\end{align*}
We have that $V(\triangle, \triangle, -\triangle)=V(\triangle, -\triangle, -\triangle)=cV(\triangle)$, where $c\triangle$ circumscribes $-\triangle$.  It is clear that $c=3$, as a tetrahedron circumscribes the inverted tetrahedron formed by connecting the centroids of its faces.  Therefore, $V(\triangle, \triangle, -\triangle)=V(\triangle, -\triangle, -\triangle)=3V(\triangle)$, and the volume expression expands to
$$[\alpha^3+\frac{9}{2}\alpha^2(1-\alpha)+\frac{9}{4}\alpha(1-\alpha)^2+\frac{1}{8}(1-\alpha)^3]V(\triangle).$$
This function is maximized at $\alpha=\frac{1+2\sqrt{14}}{11}\approx0.77121003$, which generates the maximum volume approximately 0.137103138 and the volume ratio approximately 1.1633587.  

\end{proof}
Note that this is better than the ratio for the previous example.  This case and the two-dimensional case, which is known, suggest the following conjecture.

\begin{conjecture}
In any dimension $n$, the best volume ratio is generated by the Minkowski interpolation of a simplex and an inverted simplex. 
\end{conjecture}

\section{Future Research}
\label{sec:futureresearch}

The question of whether our examples indeed generate the best volume ratio is still unsolved.  The current ratio bounds are a lot bigger than our calculations, although patterns certainly suggest that we have calculated the highest ratio.  In addition, although we only investigated examples involving Minkowski sums in our paper, this does not preclude the existence of any examples that do not require Minkowski sums.  

\section{Acknowledgments}
I am grateful to Tanya Khovanova for mentoring me in this project, for suggesting many impressive ideas, and for teaching me many of the technical skills required for the calculations and graphics for this paper.  I am also grateful to Daniel Klain for teaching me so much of the background necessary for beginning my research, for encouraging me with his enthusiasm, and for introducing me to this wonderful topic.  I am also grateful to the PRIMES program for this unique opportunity.

\end{document}